\documentclass[reqno]{amsart}

\usepackage[foot]{amsaddr} 

\usepackage{amssymb,amsmath,amsthm,amstext,amsfonts}
\usepackage[dvips]{graphicx}
\usepackage{psfrag}
\usepackage{url}
\usepackage{amsmath,amstext,amsthm,amsfonts}
\usepackage{color}
\usepackage{xcolor} 

\usepackage[colorlinks=true, linkcolor=blue, urlcolor=red, citecolor=blue, hyperindex]{hyperref}

\makeatletter \@addtoreset{equation}{section} \makeatother

\renewcommand\thetable{\thesection.\@arabic\c@table}

\theoremstyle{plain}
\newtheorem{maintheorem}{Theorem}

\newtheorem{maincorollary}{Corollary}

\newtheorem{mainquestion}{Question}

\newtheorem{theorem}{Theorem}[section]
\newtheorem{proposition}[theorem]{Proposition}
\newtheorem{lemma}[theorem]{Lemma}
\newtheorem{corollary}[theorem]{Corollary}

\theoremstyle{definition} \theoremstyle{remark}
\newtheorem{remark}[theorem]{Remark}
\newtheorem{example}[theorem]{Example}
\newtheorem{definition}[theorem]{Definition}
\newtheorem{problem}{Problem}

\newtheorem{question}{Question}
\newtheorem{conjecture}{Conjecture}

\newcommand{\Sc}{\mathbb{S}}
\newcommand{\Lo}{\mathcal{L}}

\newcommand{\C}{\mathbb{C}}

\newcommand{\N}{\mathbb{N}}
\newcommand{\R}{\mathbb{R}}

\newcommand{\cE}{\mathcal{E}}

\def\ds{\displaystyle}

\begin{document}

\title{Phase transitions for transitive local diffeomorphism with break points on the circle and H\"older continuous potentials}

\author{ Thiago Bomfim$^{\ast}$ and Afonso Fernandes}

\address{Departamento de Matem\'atica, Universidade Federal da Bahia\\
Av. Ademar de Barros s/n, 40170-110 Salvador, Brazil.}
\email{$^{\ast}$Corresponding Author: tbnunes@ufba.br}
\email{afonso$_{-}$43@hotmail.com}

\date{\today}

\begin{abstract}
It is known that if $f: \mathbb{S}^{1} \rightarrow \mathbb{S}^{1}$ is a transitive $C^{1+\alpha}$-local diffeomorphism non-invertible and non-uniformly expanding, then there is a unique parameter $t_{0} \in (0 , 1]$ such that the topological pressure function $\mathbb{R} \ni t \mapsto P_{top}(f , -t\log|Df|)$ is not analytic, in particular $f$ has a phase transition with respect  to  potential $\phi := -\log|Df|$.  
On the other hand, it is known that for continuous potentials, the topological pressure function can exhibit an infinite number of phase transitions. In this paper, we study the possibilities of the behaviour of the topological pressure function
and transfer operator 
for transitive local diffeomorphism with break points on the circle and H\"older continuous potentials.
In particular, we showed that: (1) there is an open and dense subset of continuous potentials such that if a H\"older continuous potential belongs to this subset, then it has no phase transition
and the transfer operator has the spectral gap property; (2) if a H\"older continuous potential has a phase transition, then the topological pressure function and the associated transfer operator are described.
Consequently, every H\"older continuous potential has at most two phase transitions and the set of smooth potentials such that $\mathcal{L}_{f,\phi}$ has the spectral gap property, acting on the H\"older continuous space, is dense in the uniform topology.  
Furthermore, we obtain applications for  multifractal analysis of the Birkhoff average.
\end{abstract}


\maketitle

\section{Introduction}

In Physics, the term phase transition is used mainly to describe the different states of matter: solid, liquid, and gas. During a phase transition, we often have a drastic change in properties, such as a discontinuity, resulting from external changes like temperature, pressure, or other phenomena. It corresponds to a qualitative change in the statistical properties of a dynamical system. The precise definition is not common sense; it depends on the settings or properties being studied. In dynamical systems, phase transitions can often be related, for example,  to the non-uniqueness of equilibrium states or the lack of regularity of the topological pressure function. In this work,  phase transition means that the topological pressure function is not analytic, more formally:  
 We say that a continuous dynamical system $f:M\rightarrow M$, acting on a compact metric space $M$, has a \emph{(thermodynamic) phase transition} with respect to a potential $\phi:M\rightarrow \mathbb{R}$ (or $\phi$ has a phase transition) if the topological pressure function $$\R \ni t \mapsto P_{\text{top}}(f,t\phi)$$ is not analytic at some point $t_0 \in \R$, 
 where $P_{top}(f, t\phi)$ denotes the topological pressure of $f$ with respect to $t\phi$. Note that the variational principle for the topological pressure asserts that the topological pressure 
$$
P_{\text{top}}(f, \phi) = \sup\Big\{h_{\mu}(f) + \int\phi \, d\mu : \mu \text{ is an } f \text{-invariant probability measure}\Big\},
$$
where $h_{\mu}(f)$ denotes the Kolmogorov-Sinai metric entropy of $\mu$ (see e.g. \cite{W82}).  
Whenever $\phi \equiv 0$, the topological pressure $P_{\text{top}}(f, 0)$ coincides with the topological entropy $h_{\text{top}}(f)$ of $f$, which is one of the most important topological invariants and measurements of chaotic behaviour in dynamical systems. 

Due to the works of Sinai, Ruelle and Bowen \cite{S72, Bow75, BR75}, transitive hyperbolic or expanding dynamics does not admit a phase transition with respect to H\"older continuous potentials. In particular, a dense subset of potentials has no phase transition with respect to uniform topology. From a more general point of view, there are many examples in the literature of non-uniformly hyperbolic dynamics that admit phase transitions with respect to regular potentials:
\begin{itemize}
    \item  Manneville-Pomeau maps and geometric potential \cite{Lo93}, 
    \item   a large class of interval maps with indifferent fixed point and geometric potential \cite{PS92},
    \item certain quadratic maps and geometric potential \cite{CRL13}, 
    \item porcupine horseshoes and geometric potential \cite{DGR14},
    \item  geodesic flow on Riemannian non-compact manifolds with variable pinched negative sectional curvature and suitable H\"older potential \cite{IRV18},
    \item  geodesic flow on a certain M-puncture sphere and geometric potential \cite{V17}.  

\end{itemize}
However, determining which dynamical systems admit a phase transition remains an open question. Recently, by Bomfim and Carneiro \cite{BC21} proposed the following problem: 

\begin{problem}\label{probA}
What is the mechanism responsible for the existence of phase transitions for dynamics that are $C^{1}$-local diffeomorphisms, with $h_{top}(f) > 0$, with respect to H\"older continuous potentials?
\end{problem}

As a first step in this direction, we have an answer from \cite{BC21} for the previous problem when $M=\mathbb{S}^{1}$. They proved an effective thermodynamic phase transition:

\begin{corollary}[\cite{BC21}]
Let $f:\mathbb{S}^{1}\rightarrow\mathbb{S}^{1}$ a transitive non-invertible $C^{1}$-local diffeomorphism with $Df$ Hölder continuous. If $f$ is not an expanding dynamic, then there exists $t_{0} \in (0 , 1]$ such that the topological pressure function $\mathbb{R}\ni t\mapsto P_{top}(f,-t\log|Df|)$ is analytic, strictly decrease and strictly convex in $(-\infty, t_0)$, and constant equal to zero in $[t_0,+\infty)$. In particular, $f$ has a unique thermodynamic phase transition with respect to the geometric potential $-\log|Df|$.  
\end{corollary}


Understanding the associated transfer operator was a fundamental step in proving the previous result. This operator is fundamental to studying thermodynamic quantities and obtaining equilibrium states and their properties.
We will now recall this important concept in dynamical systems; for more details, see e.g. \cite{S12} or \cite{PU10}.  We define the \textit{Ruelle-Perron-Frobenius operator} or \textit{transfer operator}, which acts on function spaces, as the following:

Let $f : M \rightarrow M$ be a local homeomorphism on a compact and connected manifold. Given a complex continuous function $\phi: M \rightarrow \C$ , define the Ruelle-Perron-Frobenius operator or transfer operator $\mathcal{L}_{f, \phi}$ acting on functions $g : M \rightarrow \C$ as following:
 $$
 \mathcal{L}_{f,\phi}(g)(x) := \sum_{f(y) = x}e^{\phi(y)}g(y). \footnote{When the context is clear, for notational simplicity, we 
may omit the dependence of the transfer operator on $f$.}
 $$

 Generally, one studies this operator acting on a Banach space $E$ that is dense in $C^0(X,\C)$. Classical thermodynamic results for sufficiently \textit{chaotic} dynamics derive from good \textit{spectral} properties of this operator. The transfer operator for dynamics such as \textit{expanding} maps is shown to have the \textit{spectral gap property} for a large set of potentials (e.g. all Hölder continuous functions), which we recall now:

 Given $E$ a complex Banach space and $\Lo:E \to E$ a bounded linear operator, we say that $\Lo$ has the (strong) \emph{spectral gap property} if there exists a decomposition of its spectrum $\text{sp}(\Lo) \subset \C$ as follows: $\text{sp}(\Lo)=\{\lambda_1\}\cup \Sigma_1$ where $\lambda_1>0$ is a leading eigenvalue for $\Lo$ with one-dimensional associated eigenspace and there exists $0<\lambda_0<\lambda_1$ such that $\Sigma_1 \subset \{z\in\C: |z|<\lambda_0\}.$

  When $f$ is a mixing expanding or hyperbolic dynamical system and $\phi$ is a suitable potential, the thermodynamic properties can be recovered via the transfer operator $\Lo_{f,\phi}$. That is possible by the fact that the transfer operator $\Lo_{\phi}$ has the spectral gap property acting on a suitable Banach space, and it can be shown that $f$ does not have a phase transition with respect to such suitable potentials (see e.g. \cite{PU10}). 
 It is well known that from a dynamical system having the spectral gap property, we can deduce many important statistical properties of thermodynamic quantities (see e.g. \cite{Ba00,GL06,BCV16,BC19}).\\

Again, Bomfim and Carneiro \cite{BC21} formulated the following conjecture:
 \begin{conjecture}\label{conjA}
Let $f : M\rightarrow M$ be a $C^{2}$-local diffeomorphism on a compact Riemannian manifold $M$. Suppose $f$ is not a uniformly expanding map or uniformly hyperbolic diffeomorphism. In that case, there is a suitable potential $\phi$ such that $\mathcal{L}_{f, \phi}$ does not have the spectral gap property acting on a suitable Banach space. (H\"older continuous, smooth functions or distributions space).
\end{conjecture}

When $M=\mathbb{S}^{1}$ is the circle, we have an answer from \cite{BC21} for the previous conjecture and, as a consequence, proved the previous corollary. 

\begin{remark}
Given $r\geq 1$ and a integer $\alpha\in (0,1]$ we denote by $C^{r}(\mathbb{S}^{1},\mathbb{C})$ and $C^{\alpha}(\mathbb{S}^{1},\mathbb{C})$ the Banach space of $C^{r}$ functions and $\alpha$-Hölder continuous complex functions whose domain is $\mathbb{S}^{1}$, respectively. Furthermore, we denote by $BV(\Sc^{1})$ the Banach space of bounded variation functions whose domain is $\mathbb{S}^{1}$.
\end{remark}

\begin{theorem}[\cite{BC21}]
Let $E=C^{\alpha}(\mathbb{S}^{1},\mathbb{C})$ or $C^{r}(\mathbb{S}^{1},\mathbb{C})$ and let $f:\mathbb{S}^{1}\rightarrow\mathbb{S}^{1}$ be a transitive not invertible $C^1$-local diffeomorphism with $Df\in E$. If $f$ is not expanding, then there exists $t_0\in (0,1]$ such that the transfer operator $\mathcal{L}_{f,-t\log|Df|}$ has  the spectral gap property for all $t<t_0$ and has no spectral gap property for $t\geq t_0$, acting on $E$.  \end{theorem}

Note that in \cite{BC21} the phase transitions are obtained for the geometric potential.  Thus, a natural question is what happens to other continuous potentials. Considering only continuous potential, everything might be possible when discussing phase transitions.  In fact, 
 Kucherenko, Quas and Wolf \cite{KQW21} show us that, in the context of shifts, there is always a continuous potential $\phi$ such that the associated topological pressure function has infinite phase transitions, in other words,  $\R \ni t \mapsto P_{top}(f , t\phi)$ is not analytic for infinitely many parameters $t$.



Moreover,  Kucherenko and Quas \cite{KQ22} presented a method to construct a continuous potential whose pressure function
coincides with any prescribed convex Lipschitz function.


Thus, it is natural to consider a regular potential for a global understanding of topological pressure. Based on the previous discussion, we propose the following question:

\begin{mainquestion}\label{questA}
For a regular potential $\phi$:
Can we describe the topological pressure function $t\rightarrow P_{top}(f,t\phi)$?
Is there a dense subset of potentials such that it has a (or has no) phase transition? 
\end{mainquestion}

On the other hand, as we have already mentioned, a way to prove the non-existence of phase transitions is to obtain the spectral gap property for the associated transfer operator. The following result by Kloeckner \cite{Kl20}, shows that for a Manneville-Pomeau-like map, the set of Hölder continuous potentials such that the associated transfer operator has the spectral gap property is dense in the uniform topology:

\begin{theorem}{(Density of spectral gap )}
Let $T$ be a degree $k$ self-covering of the circle with a neutral fixed point $0$, uniformly expanding outside each neighbourhood of $0$. For any $\alpha\in [0,1)$, let $V$ be the linear space of $\mathcal{C}^{\alpha}$ potentials which are constant near the neutral point. Then for all $\phi\in V$ the transfer operator $\mathcal{L_{T,\phi}}$ acting on $\mathcal{C}^{\alpha}(\mathbb{T})$ has the spectral gap property. Furthermore, for all $\gamma\in (0,\alpha)$, $V$ is dense in $\mathcal{C}^{\alpha}(\mathbb{T})$ for the $\gamma$-Hölder norm.    
\end{theorem}

Therefore, we propose the following question:

\begin{mainquestion}\label{questB}
Is there a dense set of regular potentials such that the associated transfer operator with it has the spectral gap property, acting on a suitable Banach space?
\end{mainquestion}

In fact, in the enumerable shift case, an analogous question
had already been proposed by Cyr and Sarig \cite{CS09}.

 \smallskip
The main goal of this paper is to answer Questions \ref{questA} and \ref{questB} for transitive local diffeomorphism with break points on the circle\footnote{See definition in Section~\ref{sec:set}} and H\"older continuous potentials. In our context, we prove that it is common that a H\"older continuous potential has no phase transition (see Theorem \ref{maintheogoo}). We also characterized the existence of phase transitions through the associated transfer operator and the convexity of the topological pressure function (see Theorem \ref{maintheA}). Consequently, we extend Kloeckner's result \cite{Kl20} about the density of the spectral gap property. Moreover, we describe the behaviour of the topological pressure function and transfer operator for H\"older continuous potentials (see Theorem \ref{maintheorB}). In particular, every H\"older continuous potential admits at most two phase transitions. Lastly, using the understanding of the topological pressure function and the transfer operator, we obtain an excellent description of the large deviations principle and multifractal analysis for the Birkhoff average of H\"older continuous observable (see Corollaries \ref{maincoroB} and \ref{maincoroC}).

\smallskip
This paper is organized as follows. Section~\ref{Statement of the main results} is devoted to the statement of the main results on phase transitions and applications to the large deviations principle and multifractal analysis.
In Section~\ref{prelim}, we recall the notions and some results on Lyapunov exponents, Ergodic optimization and the Transfer operator.
The proofs of the main results appear in 
Section~\ref{Proofs}. Finally, in Section~\ref{stratpro} we pose some questions. 
%
%

\section{Definitions and statement of the main results}\label{Statement of the main results}

\subsection{Setting}\label{sec:set}

Throughout the article we will consider transitive $C^{1}$-local diffeomorphism with break points $f : \Sc^{1} \rightarrow \Sc^{1}$, in other words,  $f$ is continuous, transitive and there exist closed arcs $J_{1} , \ldots, J_{k} \subset \Sc^{1},$
with $k > 1$, such that:
\begin{itemize}
    \item $\Sc^{1} = \bigcup_{m = 1}^{k} J_{m}$ and the arcs $J_{m}$ have disjoint interiors;
    \item $f : J_{m} \rightarrow \Sc^{1} $ is a $C^{1}$-local diffeomorphism, $f(J_{m}) = \Sc^{1}$ and $f_{|int(J_{m})}$ is injective, where $int(J_{m})$ denotes the interior of $J_{m}$;
    \item if $f(p) = p$ then the derivative $Df(p)$ is well defined.
\end{itemize}



\begin{remark}\label{preperiodic}
Note that, in our context, $Df$ may have discontinuity points. On the other hand, if $x$ is a discontinuity point for $Df$ then $x$ is an endpoint for any intervals $J_{i}$; moreover, all these are mapped on the same fixed point. In particular, since it is assumed that $Df$ is continuous at fixed points, then discontinuity points for $Df$ are pre-periodic and
cannot be periodic.
\end{remark}

We will take the Banach space $E = C^{\alpha}(\Sc^{1} , \C), BV(\Sc^{1}),$ or $C^{r}(\Sc^{1}, \C)$, for $\alpha \in (0 , 1]$ and $r \geq 1$ integer. When $E = C^{r}(\Sc^{1}, \C)$, we will assume that the dynamics $f : \Sc^{1} \rightarrow \Sc^{1}$ is a transitive $C^{r}$-local diffeomorphism.

Follows of \cite{CM86} that in our context $f$ is topologically conjugate to a mixing expanding dynamics $g(x) := deg(f)x \text{ mod } 1$, where $deg(f)$ is the topological degree of $f$. In particular, $f$ is topologically exact, strongly transitive, and has the periodic specification property.

\begin{remark}
We say that an observable $\phi : \Sc^{1} \rightarrow \R$ is cohomologous to a constant if there exists a continuous function $u:\Sc^{1} \rightarrow \R$ and a constant $K \in \R$ such that $ \phi = K + u\circ f - u$. Note that in this case, we have an excellent understanding of the thermodynamic information of $\phi$; in fact, $P_{top}(f , t\phi) = Kt + h_{top}(f)$. Since $f$ has the specification property in our context, applying  \cite[Lemma 1.9]{T10}, then not being cohomologous to a constant is an open and dense property in the uniform topology. 
\end{remark}

\

\subsection{Main results}

Our first result will show us that, typically, the potential
has no phase transition:

\begin{maintheorem}\label{maintheogoo}
There exists an open and dense subset $\mathcal{H} \subset C(\Sc^{1} ,\R)$, in the uniform topology, such that if $\phi \in \mathcal{H}$ is H\"older continuous
then $ \phi $ has no thermodynamic phase transition and 
 $t \mapsto P_{top}(f , t\phi)$ is strictly convex.
\end{maintheorem}

Follows directly from the previous theorem:

\begin{maincorollary}\label{corF}
\
\begin{enumerate}
    \item $\{ \phi : \Sc^{1} \rightarrow \R \text{ smooth } ;\; \phi \text{ has no thermodynamic phase transition and } t \mapsto P_{top}(f , t\phi) \text{ is strictly convex } \}$ is dense in the uniform topology;\\
    \item $\{\phi  : \Sc^{1} \rightarrow \R \text{ H\"older continuous 
    } ;\;
    \phi \text{ has a thermodynamic phase transition }\}$ is not dense in the uniform topology.
\end{enumerate}
\end{maincorollary}

We also obtain characterizations for thermodynamic phase transition in the case of H\"older continuous potentials. 

\begin{maintheorem}\label{maintheA}
Given $\phi \in E$ such that $\phi : \Sc^{1} \rightarrow \R$ is a H\"older continuous potential,
 then they are equivalent:

\begin{enumerate}
    \item $\phi$ has no thermodynamic phase transition;\\
    \item $\phi$ has no spectral phase transition, i.e., $\mathcal{L}_{f, t\phi}$ has the spectral gap property acting on $E$, for all $t \in \R$;\\
\end{enumerate}
 Furthermore, if $\phi$ is not cohomologous to a constant, then all the previous items are equivalent to the
 \begin{itemize}
     \item[(3)] topological pressure function $\R \ni t \mapsto P_{top}(f , t\phi)$ is strictly convex.\\
 \end{itemize}
 Moreover, all the previous items imply that $t\phi $ has a unique equilibrium state for all $t \in \R$.
\end{maintheorem}

 It is important to point out that, by Leplaideur \cite{L15}, there is an example of a continuous potential defined on a mixing subshift of finite type such that the pressure function is analytic, but the uniqueness of the equilibrium state fails.   
 
It follows from the previous results that
 $$\{\phi \in E : t\phi \text{ has a unique equilibrium state for all } t \in \R \}$$ is dense, in the uniform topology. Furthermore, $$\{\phi \in E : \mathcal{L}_{f, t\phi|E} \text{ has the spectral gap property for all } t \in \R\}$$ is dense in the uniform topology. It extends an analogous result obtained in \cite{Kl20}, in your case the dynamic is Manneville-Pomeau-like and the Banach space $E$ is $C^{\alpha}(\Sc^{1} , \C)$.

The following result shows us that, for H\"older continuous potentials, we have a good thermodynamic understanding even if a phase transition occurs. It will be a counterpoint to the results of \cite{KQW21,KQ22}.

\begin{maintheorem}\label{maintheorB}
Let $\phi \in E$ be such that $\phi : \Sc^{1} \rightarrow \R$ is a H\"older continuous potential.
If $\phi$ has a thermodynamic phase transition then there exist $ -\infty \leq t_{1} < t_{2} \leq +\infty $, with at least one of them belonging to $\R$, such that:
\begin{enumerate}
    \item the topological pressure function $\R \ni t \mapsto P_{top}(f , t\phi)$ is analytic and strictly convex  in $(t_{1} , t_{2})$, and linear in $(-\infty , t_{1}] \cup [t_{2} , +\infty]$;\\
    \item $\mathcal{L}_{f, t\phi}$ has the spectral gap property acting on $E$, for all $t \in (t_{1} , t_{2})$, and $\mathcal{L}_{f, t\phi}$ has no spectral gap property acting on $E$, for all $t \in (-\infty , t_{1}] \cup [t_{2} , +\infty]$.
\end{enumerate}
\end{maintheorem}


It follows from the previous result that every  H\"older continuous potential admits at most two thermodynamic phase transitions.

\begin{definition}
We say that a transitive $C^{1}$-local diffeomorphism with break points $f$ is Manneville-Pomeau-like if $f$ is expanding except in the unique fixed point, more formally,  $f(0) = 0$, $Df(0) = 1$ and $|Df(x)| > 1$ for all $x\neq 0$ or $1$, when the derivative is well defined.
\end{definition}

\begin{maincorollary}\label{maincoroA}
If $f$ is Manneville-Pomeau-like then every H\"older continuous potential $\phi : \Sc^{1} \rightarrow \R$
admits at most one thermodynamic phase transition.
\end{maincorollary}

\

\subsection{Applications}

Due to the previous theorem, we obtain a good description of the large deviations and multifractal analysis problems.

\subsubsection{Large deviations}\label{sect:LDR}

In the nineties, Young, Kifer and Newhouse \cite{Y90,K90,KN91} addressed the question of the velocity of convergence of ergodic averages, establishing a connection between the theory of large deviations in probability and dynamical systems. Since then, this topic has attracted attention from the mathematical community (see e.g. \cite{MN08,YC16,BV19,TT20}).

In our context, since  $f$ has the specification property and is expansive, then the maximum entropy measure $\mu_{0}$ has the Gibbs property, by \cite{Bow74}. Thus, given the continuous observable $\phi : \Sc^{1} \rightarrow \R$, by \cite{Y90}, the  large deviation rate function of $\mu_{0}$ is well defined:
$$
LD_{\phi, a, b} := \lim\frac{1}{n}\log \mu_{0}\Big(x \in \Sc^{1} : \frac{1}{n}\sum_{i=0}^{n-1}\phi (f^{i}(x)) \in [a ,b] \Big) =
$$
$$-h_{top}(f) + \sup\Big\{h_{\nu}(f) : \int \phi d\nu \in [a , b]\Big\}.
$$
Note that we are interested only when $[a , b]$ intersects the Birkhoff  spectra of $\phi$:
$$
S_{\phi}:= \Big\{\alpha \in \R : \exists x \in \Sc^{1} \text{ with }\lim\frac{1}{n}\sum_{i=0}^{n-1}\phi (f^{i}(x)) = \alpha\Big\}.
$$
It follows from the specification property that Birkhoff spectra satisfy
$$
S_{\phi} = \Big\{\int \phi d\nu : \nu \in \mathcal{M}_{1}(f)\} = \{\int \phi d\nu : \nu \in \mathcal{M}_{e}(f)\Big\}
$$
and is a compact interval (see \cite{T10}).

Thus, the natural domain for studying the regularity of the function $(a , b) \mapsto LD_{\phi, a, b}$ will be 
$$
\Delta := \{(a , b) \in S_{\phi} \times S_{\phi} : a \leq b\}.
$$

\begin{maincorollary}\label{maincoroB}
If $\phi : \Sc^{1} \rightarrow \R$ is a H\"older continuous potential 
then the large deviations rate function $\Delta \ni (a , b) \mapsto LD_{\phi, a,b}$ is convex and we can decompose $\Delta$ in regions $\Delta_{1}, \Delta_{2}, \Delta_{3}$ such that the large deviations rate function is affine in $\Delta_{1}$, constant and equal the $0$ in $\Delta_{2}$,  analytic and strictly convex in $\Delta_{3}$.
\end{maincorollary}

It is important to note that the regions $\Delta_{1}, \Delta_{2}$ and $\Delta_{3}$ obtained in the previous result will be explicit, moreover $\Delta_{1}$ can be an empty set.

\

\subsubsection{Multifractal analysis}

In multifractal analysis, we study invariant sets and measures with a multifractal structure. We want to measure the size of those sets in the sense of Hausdorff dimension or topological entropy, for instance. This study can be traced back to Besicovitch and has received contributions from many authors in recent years (see e.g. \cite{JR11,C14,IJT16,BV17} ).

In our context, we are interested in the case of Birkhoff averages.
Given a continuous observable $\phi : \Sc^{1} \rightarrow \R$ and $a,b \in \R$ define:
$$
X_{\phi , a , b}:= \Big\{x \in \Sc^{1}: a \leq \liminf \frac{1}{n}\sum_{i=0}^{n-1}\phi (f^{i}(x)) \leq \limsup \frac{1}{n}\sum_{i=0}^{n-1}\phi (f^{i}(x)) \leq b\Big\}.
$$
We aim to describe these sets from the point of view of topological pressure. Just like before, the natural domain of the function $(a , b) \mapsto h_{X_{\phi , a , b}}$ is $\Delta$, where $h_{Z}$ denotes the topological entropy restricted to the set $Z$.

\begin{maincorollary}\label{maincoroC}
    If $\phi : \Sc^{1} \rightarrow \R$ is a H\"older continuous potential 
    then
    the large deviations rate function $\Delta \ni (a , b) \mapsto h_{X_{\phi, a,b}}(f)$ is concave, affine in $\Delta_{1}$, constant and equal the $h_{top}(f)$ in $\Delta_{2}$,  analytic and strictly concave in $\Delta_{3}$.
\end{maincorollary}

\


\section{Preliminaries}\label{prelim}

In this section, we provide some definitions and preparatory results needed to prove the main results.  We first recall some concepts and results related to Lyapunov exponents and metric entropy (Subsection~\ref{sec:A}), Ergodic optimization (Subsection~\ref{sec:B}) and the Transfer operator (Subsection~\ref{sec:C}).

\subsection{Ergodic Theory}\label{sec:A}

Now we state some definitions,
 notations and results from Ergodic Theory (see e.g. \cite{OV16} for more details).
 
The \emph{Birkhoff's Ergodic Theorem}  relates time and space averages of a given potential $\phi: X \to \R,$ more formally: 

\begin{theorem}(Birkhoff) Let $f:X\to X$ be a measurable transformation and $\mu$ be an $f$-invariant probability. Given any integrable function $\phi: X \to \R$, the limit:
$$\bar{\phi}(x)=\lim_{n\to\infty} \dfrac{1}{n}\sum_{j=0}^{n-1} \phi(f^j(x))$$
exists in $\mu$-a.e. $x \in X$. Furthermore, the function $\Bar{\phi}$ defined this way is integrable and satisfies
$$\int\bar{\phi}(x)d\mu(x)=\int\phi(x)d\mu (x).$$
Additionally, if $\mu$ is $f$-ergodic, then $\Bar{\phi}\equiv\int \phi d\mu$ for $\mu$-a.e..
\end{theorem}

We denote the $f$-invariant probabilities space by $\mathcal{M}_{1}(f)$  and  the $f$-invariant and ergodic probabilities space by $\mathcal{M}_{e}(f)$.

The \emph{Lyapunov exponents} are important quantities that can sometimes be computed from time averages. These exponents translate the asymptotic rates of expansion and contraction of a dynamical system.
For our context, $f$ is a piecewise monotone on the circle, the Lyapunov exponents are defined simply as:

$$\lambda(x)=\lim_{n\to\infty} \dfrac{1}{n}\log|Df^{n}(x))|=\lim_{n\to\infty}\sum_{j=0}^{n-1}\dfrac{1}{n}\log|Df(f^j (x))|,$$
whether the limit exists.
The Lyapunov exponents $\lambda$ coincide with the time average for the observable $\log|Df|$  in each point $x$ where the limit exists.
On the other hand, given  $\mu \in \mathcal{M}_{e}(f)$, by Birkhoff's Ergodic Theorem, we have
$$\lambda(x)=\int \log|Df|d\mu \text{ for } \mu\text{-a.e. } x.$$ Thus, given $\mu \in \mathcal{M}_{1}(f)$ we define the Lyapunov exponent for this measure $\lambda(\mu):=\int \log|Df|d\mu$.

We recall an estimate for the metric entropy through the Lyapunov exponents; for proof in our context, see \cite{H91}:

\begin{theorem}(Margulis-Ruelle inequality)
If $\mu \in \mathcal{M}_{e}(f)$ then $h_\mu(f)\leq \max\{0 , \lambda(\mu)\}.$
\end{theorem}

\

\subsection{Ergodic optimization}\label{sec:B}

We will see that the "good" potentials maximize invariant probabilities with positive Lyapunov exponents. Therefore, we recall some key concepts and results in ergodic optimization.

Let $T:X\rightarrow X$ be a continuous transformation of a compact metric space. For each continuous function $\phi: X\rightarrow\mathbb{R}$ we define the maximum ergodic average 

\begin{align*}
    \beta(\phi):=\sup\limits_{\mu\in \mathcal{M}_1(T)}\int \phi d\mu=\sup\limits_{x\in M}\limsup\limits_{n\rightarrow\infty}\frac{1}{n}\sum\limits_{k=0}^{n-1}\phi(T^{k}(x))
\end{align*}
and the set of all maximizing measures of $\phi$:

\begin{align*}
    \mathcal{M}_{\max}(\phi):=\Big\{\mu\in\mathcal{M}_1(T);\int \phi d\mu=\beta(\phi)\Big\}.
\end{align*}

The problem of understating which orbits or measures attain the maximum ergodic average is known as \emph{Ergodic Optimization}; for a review about the subject, see e.g. \cite{J19} 

We will need a result of ergodic optimization; for the proof see \cite[Lemma 3.3]{M10}:

 \begin{theorem}\label{morris}
Let $T:X\rightarrow X$ be a continuous transformation of a compact metric space. Suppose that $\mathcal{U}$ is a dense  subset of $\mathcal{M}_{e}(T)$. Then the set
\begin{align*}
    U:=\Big\{\phi\in C(M, \R);\mathcal{M}_{\max}(\phi)\subset \mathcal{U} \Big\}
\end{align*}
is dense in $C(M, \R)$. 
\end{theorem}


\subsection{Transfer operator}\label{sec:C}

In this section,  we recall some properties of the transfer operators. For more details on the transfer operator, see e.g. \cite{S12} or \cite{PU10}.

In what follows, given $A : E \rightarrow E$ a bounded linear operator, we denote its spectral radius by $\rho(A)$.

Let $f$ be as in our context. Given a complex continuous function $\phi: \Sc^{1} \rightarrow \C$ , we define the Ruelle-Perron-Frobenius operator or \emph{transfer operator} $\mathcal{L}_{f, \phi}$ acting on functions $g : \Sc^{1} \rightarrow \C$ this way:
 $$
 \mathcal{L}_{f,\phi}(g)(x) := \sum_{f(y) = x}e^{\phi(y)}g(y). 
 $$

If $\phi$ is a real continuous function, via Mazur's Separation Theorem, since $\Lo_{f,\phi}$ is a positive operator then it has $\rho(\Lo_{f,\phi}|_{C^0})$ as an eigenvalue for its dual operator, that is, there exists a probability $\nu_\phi$ with $(\Lo_{f,\phi}|_{C^{0}})^* \nu_{\phi} = \rho(\Lo_{f,\phi}|_{C^0}) \nu_\phi$. This probability is often referred to as a \emph{conformal measure}.

Generally, verifying that a specific linear operator has the spectral gap property is complicated. Sometimes, it is convenient to consider a weaker spectral property, for instance, the called \emph{quasi-compactness}:
 Given $E$ a complex Banach space and $A:E \to E$ a bounded linear operator, we say that $A$ is quasi-compact if there exist $0<\sigma<\rho(\mathcal{A})$ and a decomposition of $E=F\oplus H$ as follows: $F$ and $H$ are closed and $\mathcal{A}$-invariant, $\dim F < \infty$, $\rho(A_{_|F})>\sigma$ and $\rho(\mathcal{A}_{|H}) \leq \sigma$.

To obtain quasi-compactness, we use an alternative equivalent definition for it via the \emph{essential spectral radius}:
Given $E$ a complex Banach space and $A:E \to E$ a bounded linear operator, define the essential spectral radius:
$$\rho_{ess}(A):=\inf\{r>0; \;\text{sp}(A)\setminus \overline{B(0,r)} \text{ contains only eigenvalues of finite multiplicity}\}$$
\noindent Thus, quasi-compactness is equivalent to having $\rho_{ess}(A) < \rho(A)$, and so estimates on the essential spectral radius and the spectral radius will be crucial.

\subsection{Analytic functions on Banach spaces}
In this section, we recall the notion of analyticity for functions on Banach spaces. Let $E , F$ be Banach spaces and denote by $\mathcal{L}^{i}_{s}(E , F)$ the space of continuous and symmetric $i$-linear transformations from
$E^{i}$ to $F$. 

\begin{definition}
Let $U\subset E$ be an open subset. We say the function $g : U \subset E \rightarrow F$
is \emph{analytic} if for all $x \in U$ there exists $r > 0$ and for every $i\ge 1$ there exists $P_{i} \in \mathcal{L}^{i}_{s}(E , F)$ (depending on $x$) such that
$$
g(x + h) = g(x) + \sum_{i=1}^{\infty}\frac{P_{i}(h, \ldots, h)}{i!}
$$ for all $h \in B(0 , r)$ and the convergence is uniform.
\end{definition}

Analytic functions on Banach spaces have similar properties to those of real and complex analytic functions.
For instance, if $g : U \subset E \rightarrow F$ is analytic then $g$ is $C^{\infty}$ and for every $x\in U$ one has
$P_{i}=D^{i}f(x)$. For more details, see e.g. \cite[Chapter~12]{Cha85}.






\section{Proof of the main results}\label{Proofs}


\subsection{Hyperbolic potentials}

Following \cite{IR12}, given a continuous dynamical system $T : X \rightarrow X$ acting on a compact metric space and a continuous potential $\phi : X \rightarrow \R$, we say that $\phi$ is hyperbolic if $$ \sup_{\mu \in \mathcal{M}_{1}(T)}\int \phi d\mu < P_{top}(T , \phi).$$

\begin{remark}\label{remmax}
Note that $\phi$ is hyperbolic if and only if there is no equilibrium state $\mu \in \mathcal{M}_{1}(T)$ with respect to $\phi$ such that $h_{\mu}(T) = 0$. 
\end{remark}

For certain dynamics, such as multimodal or rational maps, the hyperbolic conditions on the potential ensure good thermodynamic properties, including the uniqueness of equilibrium states and the spectral gap property of the associated transfer operator (see  \cite{IR12,HR14}).

\begin{proposition}\label{propha}
If $h_{top}(T) > 0$ and $\phi$ is not hyperbolic then:

\begin{enumerate}
    \item $\phi$ has a maximizing measure with zero entropy, i.e., there is $\eta \in \mathcal{M}_{1}(T)$ such that $\ds\int \phi d\eta = \sup_{\mu \in \mathcal{M}_{1}(T)}\int \phi d\mu$ and $h_{\eta}(T) = 0$;
    \item there is $t_{0} \in (0 , 1]$ such that $ P_{top}(T , t\phi) = t\int \phi d\eta $, for all $t \geq t_{0}$, in particular $\phi$ has a freezing phase transition in $t_{0}$.
\end{enumerate}
\end{proposition}
\begin{proof}
$(1)]$ Let $\eta \in \mathcal{M}_{1}(T)$ be, with $\int \phi d\eta = P_{top}(T , \phi)$. Hence $h_{\eta}(T) = 0$. Given $\mu \in \mathcal{M}_{1}(T)$, we have that $\int \phi d\mu \leq h_{\mu}(T) + \int \phi d\mu \leq P_{top}(T , \phi)= h_{\eta}(T) + \int \phi d\eta$. Thus  $\int \phi d\eta \geq \int \phi d\mu$, for all $\mu \in \mathcal{M}_{1}(T)$.\\

$(2)]$ Let $t \geq 1$. Given $\mu \in \mathcal{M}_{1}(T)$, we have that
$$
h_{\mu}(T) + t\int \phi d\mu \leq h_{\mu}(T) + \int \phi d\mu + (t-1)\int \phi d\mu \leq P_{top}(T , \phi) + (t-1)\int \phi d\mu \leq 
$$
$$
\int \phi d\eta  + (t-1)\int\phi d\eta = t\int\phi d\eta \Rightarrow P_{top}(T , t\phi) = t\int\phi d\eta.
$$
Define $t_{0} := \inf\{ t \in (0 , 1]: t\phi \text{ is not hyperbolic }\}$. Since  $h_{top}(T) > 0$, then $t_{0} \in (0 , 1]$. Furthermore, for what we have already proved,  $P_{top}(T , t\phi) = t\int \phi d\eta $, for all $t \geq t_{0}$. If $\R \in t \mapsto P_{top}(T , t\phi)$ was analytic then $P_{top}(T , t\phi) = t\int \phi d\eta $, for all $t \in \R$. In particular $h_{top}(T) = P_{top}(T , 0\phi) = 0$, which is absurd.
\end{proof}

\subsection{Expanding potentials}


As a first result, we show that the Lyapunov exponent varies continuously in the space of invariant probability measures. 
\;
\begin{lemma}\label{expconti}
    The function $\mathcal{M}_1(f)\ni \mu\mapsto \int\log|Df|d\mu$ is continuous.
  \end{lemma}
  \begin{proof}
    In fact, as $Df$ is a function of bounded variation and $f$ does not have critical points, $\log|Df|$ is bounded, in particular $|\log|Df||\leq K$. By assumption, $Df$ has at most a finite number of discontinuity points $D:=\{x_1,...,x_k\}$. We must show that for any invariant probability $\mu\in \mathcal{M}_{1}(f)$ we have $\mu(\{x_i\})=0$ for each discontinuity point $x_i$.  We claim that for any $m>n\in\mathbb{N}$, $f^{-n}(x_i)\cap f^{-m}(x_i)=\emptyset$. Otherwise, let $x\in f^{-n}(x_i)\cap f^{-m}(x_i)$ be, then we would have $f^{m-n}(x_i)=x_i$ which means that $x_i$ would be a periodic point. However, it follows from Remark \ref{preperiodic}  that every $x_i$ is a pre-periodic point and cannot be a periodic point. From that, we conclude by the invariance of $\mu$ that
    $$\mu(\bigcup\limits_{n\in \mathbb{N}}f^{-n}(x_i))=\sum\limits_{n=0}^{\infty}\mu(f^{-n}(x_i))=\sum\limits_{n=0}^{\infty}\mu(x_i).$$
So, $\mu(x_i)=0$ for every invariant probability measure $\mu$.
    Given $\delta>0$, take the closed subspace $A_{\delta}=[0,1]-\bigcup_{x_i}B(x_i,\delta)$. By Tietze's extension theorem, let $\phi_{\delta}:[0,1]\rightarrow \mathbb{R}$ be a continuous extension of $\log|Df|_{|A_{\delta}}$. Take any sequence $\mathcal{M}_1(f)\ni\mu_n\xrightarrow{w^{*}}\mu$, then by definition we have:
    $$\int \phi_{\delta} d\mu_n\rightarrow \int \phi_{\delta} d\mu.$$
    Now, for $\epsilon>0$ let $n_0\in \mathbb{N}$ be and $\delta>0$ be sufficiently small such that:
    $$|\int \phi_{\delta} d\mu_n-\int \phi_{\delta} d\mu|<\epsilon/3,\; \forall n\geq n_0$$
    and
    $$\sup\limits_{x\in [0,1]\setminus D}|\phi_{\delta}(x)-\log|Df(x)||<\epsilon/3,$$
    which means that (taking the integrals in $[0,1] \setminus D$)
    $$|\int \phi_{\delta} d\mu -\int \log|Df|d\mu|<\epsilon/3$$
    and
    $$|\int \phi_{\delta} d\mu_n -\int \log|Df|d\mu_n|<\epsilon/3.$$
\noindent Thus,
    $$|\int \log|Df|d\mu_n-\int\log|Df|d\mu|\leq |\int\log|Df|d\mu_n-\int \phi_{\delta} d\mu_n|+$$ \\
    $$|\int \phi_{\delta} d\mu_n-\int\phi_{\delta}d\mu|+|\int\phi_{\delta}d\mu-\int \log|Df|d\mu|<\epsilon.$$
    As $\epsilon$ is arbitrary, we have the result.
\end{proof}

Now, we prove that the our dynamic has no negative Lyapunov exponents.

\begin{proposition}
If $\mu \in \mathcal{M}_{1}(f)$ then $\lambda(\mu) \geq 0$.
\end{proposition}
\begin{proof}
Suppose there is $\mu \in \mathcal{M}_{1}(f)$ with $\lambda(\mu) = \int \log |Df| d\mu< 0$. Since  $f$ has periodic specification property, applying \cite{S74}, there exists $\nu = \frac{1}{k}\sum_{i=0}^{k-1}\delta_{f^{i}(p)} \in \mathcal{M}_{1}(f)$ a periodic measure such that $0> \lambda(\nu) = \frac{1}{k}
\log|Df^{k}(p)|$. In particular, $p$ will be a periodic attractor orbit, contradicting the fact that $f$ is transitive.
\end{proof}


In our context, following \cite{PV22}:

\begin{definition} 
Given a continuous potential $\phi : \Sc^{1} \rightarrow \R$, we say that $\phi$ is expanding if $$ \sup\Big\{h_{\mu}(f) + \int \phi d\mu : \mu \in \mathcal{M}_{1}(f) \text{ and } \lambda(\mu) = 0\Big\} < P_{top}(f, \phi).$$
\end{definition}

\begin{remark}
Note that $\phi$ is expanding if and only if there is no equilibrium state $\mu \in \mathcal{M}_{1}(T)$ with respect to $\phi$ such that $\lambda(\mu) = 0$. 
\end{remark}

In our context, due to inequality $h_{\mu}(f) \leq \max\{0, \lambda(\mu) \}$, we have that every hyperbolic potential is an expanding potential.

For certain multimodal or rational maps, a potential is hyperbolic if and only if it is expanding (see  \cite{IR12,HR14}). We will prove that this result also holds in our context, for regular potentials.

\begin{remark}
 It is worth noting that the equivalence between hyperbolic potential and expanding potential is not obvious, even when using Pesin's formula. It could be that an equilibrium state has zero unstable
Hausdorff dimension, thus zero entropy.
\end{remark}

We observe that, in our context, $f$ is strongly transitive; thus, the transfer operator is quasi-compact if and only if it has the spectral gap property. 
The proof is analogous to the respective result in \cite[Proposition 5.9]{BC21}:

\begin{lemma}\label{LemaEss}
Let $\phi $ be a continuous potential. If $\Lo_{f,\phi}|_E$ is quasi-compact, then it  has the spectral gap property.
\end{lemma}

The following result will connect the expanding/hyperbolic property of the potential with the spectral gap property. The proof is inspired by the work \cite{BC21}.

\begin{proposition}\label{propexp}
Let $\phi \in E$ be a continuous potential. Suppose that one of the two items occurs:
\begin{enumerate}
    \item $\phi$ is expanding and $E = C^{\alpha}(\Sc^{1} , \C)$ or $C^{r}(\Sc^{1} , \C)$;
    \item $\phi$ is hyperbolic and $E = BV(\Sc^{1})$.
\end{enumerate}
Then $\mathcal{L}_{f , \phi}$ has the spectral gap property, acting on $E$.
\end{proposition}
\begin{proof}

\

\emph{Case 1: $\phi$ is expanding and $E = C^{\alpha}(\Sc^{1} , \C)$ or $C^{r}(\Sc^{1} , \C) $.}

\

\noindent On the one hand, if $E = C^{\alpha}(\Sc^{1} , \C) $,  applying \cite[Theorem 1 and 2]{BJL96} we have that:
$$
\rho_{ess}(\mathcal{L}_{f,\phi|C^{\alpha}}) = \lim\sqrt[n]{||\mathcal{L}_{f,\phi - \alpha \log |Df|}^{n}1||_{\infty}} \text{ and }
$$
$$
\rho(\mathcal{L}_{f,\phi|C^{\alpha}}) = \max\Big\{ \rho_{ess}(\mathcal{L}_{f,\phi|C^{\alpha}}), \lim\sqrt[n]{||\mathcal{L}_{f,\phi}^{n}1||_{\infty}}\Big\}.
$$
Note that, 
in our context,
$$\lim\sqrt[n]{||\mathcal{L}_{f,\phi - \alpha \log |Df|}^{n}1||_{\infty}} = e^{P_{top}(f ,\phi - \alpha\log|Df|)} \text{ and }
$$
$$
\lim\sqrt[n]{||\mathcal{L}_{f,\phi}^{n}1||_{\infty}} = e^{P_{top}(f ,\phi)}.
$$

On the other hand, if $E = C^{r}(\Sc^{1} ,\C)$, applying \cite{CL97} we have that:
$$
\rho_{ess}(\mathcal{L}_{f,\phi|C^{r}}) \leq e^{P_{top}(f , \phi - r\log|Df|)} \text{ and }
\rho(\mathcal{L}_{f,\phi|C^{r}}) \leq  e^{P_{top}(f,\phi)}.
$$
Since there is $\nu_{\phi}$ conformal measure then $\rho(\mathcal{L}_{f,\phi|C^{0}}) \leq \rho(\mathcal{L}_{f,\phi|C^{r}})$, thus $\rho(\mathcal{L}_{f,\phi|C^{0}}) = \rho(\mathcal{L}_{f,\phi|C^{r}})$.

In both cases, there exists $k > 0$ such that:
 $$
\rho_{ess}(\mathcal{L}_{f,\phi|E}) \leq e^{P_{top}(f , \phi - k\log|Df|)} \text{ and }
\rho(\mathcal{L}_{f,\phi|E}) =  e^{P_{top}(f,\phi)}.
$$
Suppose, by contradiction, that $\mathcal{L}_{f,\phi|E}$ has no spectral gap property. By Lemma \ref{LemaEss}   we have 
$$
\rho_{ess}(\mathcal{L}_{f,\phi|E}) = \rho(\mathcal{L}_{f,\phi|E}) \Rightarrow P_{top}(f , \phi) = P_{top}(f , \phi - k\log |Df|).
$$
Since $f$ is expansive and $\mathcal{M}_{1}(f) \ni \mu \mapsto \lambda(\mu)$ is continuous, there is an equilibrium state $\eta$ with respect a $\phi - k\log|Df|$. Note that for all $f$-invariant probability $\mu$:
$$
h_{\eta}(f) + \int [\phi - k\log |Df|]d\eta = P_{top}(f ,\phi - k\log |Df|) = P_{top}(f , \phi ) \geq h_{\mu}(f) + \int \phi  d\mu.
$$
In particular, $\lambda(\eta) = 0$ and $\eta$ is an equilibrium state with respect to $\phi$. In other words, $\phi $ is not expanding, which is absurd by hypothesis. We conclude then that $\mathcal{L}_{f,\phi|E}$ has the spectral gap property.\\

\emph{Case 2: $\phi$ is hyperbolic and $E = BV(\Sc^{1})$.}

\

\noindent Applying \cite[Theorem 3]{BK90} and \cite[Theorem 3.2]{Ba00} we have that:
$$
\rho(\mathcal{L}_{f,\phi|E}) = \rho(\mathcal{L}_{f,\phi|C^{0}}) = e^{P_{top}(f,\phi)} \text{ and }
\rho_{ess}(\mathcal{L}_{f,\phi|E}) \leq \lim\sqrt[n]{||e^{S_{n}\phi}||_{\infty}}.
$$


Let us show that $\lim\sqrt[n]{||e^{S_{n}\phi}||_{\infty}} < e^{P_{top}(f,\phi)}$.
In fact; if  $\lim\sqrt[n]{||e^{S_{n}\phi}||_{\infty}} \geq e^{P_{top}(f,\phi)}$ then given $\epsilon > 0 $ there is $n_{0} \in \N$ such that $\forall n > n_{0}$ there exists $x_{n} \in \Sc^{1}$ with:
$$
\frac{1}{n}\log e^{S_{n}\phi(x_{n})} \geq P_{top}(f , \phi) - \epsilon \Rightarrow \frac{1}{n}S_{n}\phi(x_{n}) \geq P_{top}(f , \phi) - \epsilon.
$$
Take the probability $\ds\eta = \lim_{k \mapsto +\infty}\frac{1}{n_{k}}\sum_{i =0}^{n_{k}-1}\delta_{f^{i}(x_{n_{k}})}$. Note that $\eta$ is an $f$-invariant probability and
$$
\int \phi d\eta\geq P_{top}(f,\phi) - \epsilon \Rightarrow \max_{\mu \mathcal{M}_{1}(f)}\int \phi d\mu = P_{top}(f , \phi).
$$
In other words, $\phi$ is not hyperbolic, which is absurd by hypothesis.

We conclude then that $\rho(\mathcal{L}_{f,\phi|E}) > \rho_{ess}(\mathcal{L}_{f,\phi|E})$, applying the Lemma \ref{LemaEss} we have that $\mathcal{L}_{f,\phi|E}$ has the spectral gap property. 
\end{proof}

\

Since  $f$ is strongly transitive, the following result has the same proof as the respective results in \cite{BC21}:

\begin{lemma}\label{medida}
Let $\phi \in E$ be a continuous  potential. If $\mathcal{L}_{f, \phi}|_{E}$ has the spectral gap property then there exists a unique probability $\nu_{\phi}$ and $h_{\phi} \in E$ such that $ (\mathcal{L}_{f, \phi}|_{E})^{\ast}\nu_{\phi} = \rho(\mathcal{L}_{f, \phi}|_{E})\nu_{\phi}$, $\Lo_{f,\phi}h_{\phi}=\rho(\Lo_{f,\phi}|_{E}) h_{\phi}$, $h_{\phi} > 0$ and $\int h_{\phi} d\nu_{\phi} = 1$. Moreover, $supp(\nu_{\phi}) = \Sc^{1}$.
\end{lemma}

\begin{remark}
Denote the $f$-invariant probability $h_{\phi}\nu_{\phi}$ by $\mu_{\phi}$. Note that $supp(\mu_{\phi}) = \Sc^{1}.$
\end{remark}

On the other hand, define $SG(E) := \{\phi \in E : \mathcal{L}_{f,\phi|E} \text{ has the spectral gap property }\}$, analogously to \cite[Corollary 4.12]{BC21} we have that:

\begin{corollary}\label{analy}
$SG(E) \subset E$ is an open subset, and the following map is analytic:
$$
SG(E) \ni \phi \mapsto \big(\rho(\mathcal{L}_{f,\phi|E}) , h_{\phi} , \nu_{\phi}) \in \R \times E \times E^{\ast}.
$$
\end{corollary}

The following result shows that if $\Lo_{f,\phi}|_{E}$ has the spectral gap property, then the thermodynamic quantities are related to the spectral quantities.

\begin{lemma}\label{Lemapress}\label{GapAnalt}
 Let $\phi \in E$ be a continuous potential. If $\phi \in SG(E)$, then:
\begin{enumerate}
    \item $P_{top}(f,\phi)=\log \rho(\Lo_{f,\phi}|_{E})$ and $\mu_{\phi}$ is the unique equilibrium state of $f$ with respect to $\phi$;
\item $\R \ni t\mapsto P(f,t\phi)$ is analytic on a neighborhood of $1$.
\item  Suppose additionally that $\phi$ is not cohomologous to a constant then $\R \ni t \mapsto P_{top}(f , t\phi) $ is strictly convex on a neighborhood of $1$. 

\end{enumerate}

\end{lemma}
\begin{proof}
(1)] Initially we will show that $P_{top}(f,\phi)=\log \rho(\Lo_{f,\phi}|_{E})$ and $\mu_{\phi}$ is an equilibrium state of $f$ with respect to $\phi$. For the case $E = C^{\alpha}(\Sc^{1} , \C)$ or $C^{r}(\Sc^{1} , \C)$, since $f$ is strongly transitive and also admits generating partition by domains of injectivity, the proof is the same of the respective result in \cite[Lemma 5.7]{BC21}. For the case $E = BV(\Sc^{1})$ is enough apply \cite[Theorem 3]{BK90}. 

We will now show the uniqueness of the equilibrium state. For the case $E = C^{\alpha}(\Sc^{1} , \C)$ or $C^{r}(\Sc^{1} , \C)$ we could apply the respective result of \cite{PV22}, but to cover the case in which $E = BV(\Sc^{1})$ we are going to give another proof. Suppose that $\mu$ is an equilibrium state with respect to $\phi$. Fix a potential $\psi \in E \cap C(\Sc^{1} , \C)$. Since  $\Lo_{f,\phi}|_{E}$ has the spectral gap property, applying the Corollary \ref{analy}, for $t$ close to zero we have that $\Lo_{f,\phi+ t\psi}|_{E}$ has the spectral gap property.  In particular $\mu_{\phi + t\psi}$ is an equilibrium state with respect to $\phi + t\psi$.  Then:
$$
h_{\mu}(f) + \int \phi +t\psi d\mu \leq h_{\mu_{\phi + t\psi}}(f) + \int \phi d \mu_{\phi + t\psi} + t\int \psi d \mu_{\phi + t\psi} \leq
$$
$$
h_{\mu}(f) + \int \phi d\mu + t\int \psi d \mu_{\phi + t\psi} \Rightarrow t\int \psi d\mu \leq t\int \psi d \mu_{\phi + t\psi}.
$$
Take $\eta_{1 =} \lim_{t_{k} \searrow 0} \mu_{\phi + t_{k}\psi}$ and $\eta_{2}= \lim_{t_{k} \nearrow 0} \mu_{\phi + t_{k}\psi}$. Thus, $\int \psi d\eta_{2} \leq \int \psi d\mu \leq \int \psi d\eta_{1}$.
On the other hand, $\int \psi d\eta_{2} = \int \psi d\mu_{\phi} = \int \psi d\eta_{1}$. In fact; 
follows from the Corollary \ref{analy} that $\R \ni t \mapsto \int \psi d\mu_{\phi +t\psi}$ is analytic for $t$ close to zero, in particular $\R \ni t \mapsto \int \psi d\mu_{\phi +t\psi}$ is continuous in $t =0$. 
Therefore, $\int \psi d\mu = \int \psi d\mu_{\phi}$. Since $E \cap C(\Sc^{1} , \C)$ is dense in $C(\Sc^{1} , \C)$, we conclude that $\mu = \mu_{\phi}$.\\

(2)] This item is a direct consequence of the previous item and Corollary \ref{analy}.\\

(3)] Suppose that $\phi$ is not cohomologous to a constant. Since  $\mathcal{L}_{f , \phi}|_{E}$ has the spectral gap property then, analogously the ideas of the respective result in \cite[Proposition 5.8]{BC21} or \cite{BCF22}, we have the strict convexity of $(1 -\epsilon , 1 + \epsilon) \ni t \mapsto P_{top}(f , t\phi)$. For the reader's convenience, we will highlight the key points of the proof.

Let $\lambda_{\phi}$ be denoting $\rho(\Lo_{f,\phi|E})$ and fix $\psi = \phi - \int \phi d\mu_{\phi}$. Then the function $T(t):=\dfrac{\lambda_{\phi + t \psi}}{\lambda_{\phi}}$ is well defined for $t \in \R$ in a small neighbourhood of zero and is analytic. By Nagaev's method we have $\sigma^2_{f,\phi}(\psi)=-D_t^2 T(t)|_{t=0}$, where $\sigma^2:=\sigma^2_{f,\phi}(\psi)$ is the variance of the \textit{Central Limit Theorem} with respect to dynamics $f$, probability $\mu_{\phi}$ and observable $\psi$. Moreover $\sigma = 0$ if and only if  there exist $c \in \R$ and $u \in E$ such that 
$$\psi(x)= c + u\circ f(x) - u(x) , \text{ for } \mu_{\phi}\text{-a.e. } x,$$
(see \cite[Lecture 4]{S12} for more details). 
 Define $G(t) := P_{top}( f , \phi + t\psi)$, then $t \mapsto P_{top}(f , t\phi) $ is strictly convex on a neighborhood of $1$ if and only if $G$ is strictly convex on a neighborhood of $0$. Besides that $G'(0) = 
 \int \psi  d\mu_{\phi} = 0$. Note that:
 $$
 \sigma^{2} = \big(G'(0)\big)^{2} + G''(0) = G''(0).
 $$
 Suppose,  by contradiction, that $t \mapsto P_{top}(f , t\phi) $ is not strictly convex on a neighborhood of $1$, then $0 = G''(0) = \sigma^{2} = 0$. 
 Hence, by Nagaev's method, there exist $c \in \R$ and $u \in E$ such that 
$$\phi= c + u\circ f(x) - u(x) , \text{ for } \mu_{\phi}\text{-a.e. } x.$$
Since  $\mu_{\phi}$ has full support, then $\phi \equiv c + u\circ f - u$. Thus 
 $\phi $ will be cohomologous to a constant. This contradicts the hypothesis.
\end{proof}

\begin{corollary}\label{corhyexp}
Let $\phi $ be a H\"older continuous potential. Then: $\phi$ is hyperbolic if and only if  $\phi$ is expanding.
\end{corollary}
\begin{proof}
We already know that if $\phi$ is hyperbolic, then $\phi$ is expanding. Suppose that the potential $\phi$ is expanding. If $\phi$ is cohomologous to a constant then $P_{top}f, \phi) = h_{top}(f) + \int \phi d \mu$, for all $f$-invariant probability $\mu$. Since  $h_{top}(f) >0$, we have that $\phi$ is hyperbolic. If $\phi$ is not cohomologous to a constant, applying the Lemmas \ref{propexp} and \ref{GapAnalt}, then $t \mapsto P_{top}(f , t\phi)$ is strictly convex on a neighbourhood of $1$. It follows from item (2) of Proposition \ref{propha} that $\phi$ is hyperbolic.
\end{proof}

It is important to note that for continuous potentials, the previous result does not necessarily hold.

\

\subsection{Denseness of good potentials}

In this section, we will prove Theorem \ref{maintheogoo}.

It follows of the Proposition \ref{propexp} and Lemma \ref{GapAnalt} that if $\phi $ is a H\"older continuous potential and $t\phi$ is expanding for all $t \in \R$, then $f$ has no phase transition with respect to $\phi$, in addition, we obtain other thermodynamic informations. Therefore, we will show that this property is dense in the uniform topology.

Given $f$ a dynamics in our context,
we define the following set:
\begin{align*}
    \mathcal{R}:=\Big\{\phi:\mathbb{S}^{1}\rightarrow\mathbb{R}\in C^{0}; \sup\limits_{\lambda(\mu)=0}\int\phi d\mu <\sup\limits_{\lambda(\nu)>0}\int\phi d\nu\Big\}
\end{align*}

Note that $\phi \in \mathcal{R}$ if and only if there is no $\mu \in \mathcal{M}_{\max}(\phi)$ with $\lambda(\mu) = 0$. Furthermore, if $\phi \in \mathcal{R}$ then $t\phi \in \mathcal{R}$ and is not cohomologous to a constant for all $t > 0$. 

\begin{lemma}\label{claim1}
$\mathcal{R}$ is open and dense in the uniform topology.
\end{lemma}
\begin{proof}

Firstly, since   $\mathcal{M}_{1}(f) \times \mathcal{M}_{1}(f) \times C(\mathbb{S}^{1} , \R) \ni (\mu, \nu, \phi) \mapsto (\lambda(\mu) , \int \phi d\nu)$ are continuous then $\mathcal{R}$ is open. Secondly, we define
\begin{align*}
    \mathcal{U}:=\Big\{\mu\in \mathcal{M}_{e}(f);\lambda(\mu)>0\Big\}.
\end{align*}
 Thus 
 \begin{align*}
    \mathcal{R}\supset \Big\{\phi:\mathbb{S}^{1}\rightarrow\mathbb{R}\in C^{0};\mathcal{M}_{\max}(\phi)\subset \mathcal{U} \Big\}.
\end{align*}

\

 Therefore, using Morris's theorem \ref{morris}, to prove the lemma is enough to show that $\mathcal{U}$ is a dense subset in the space of ergodic probability measures.\\


\noindent\emph{Density:} Since  $f$ has positive topological entropy, applying the Margulis-Ruelle's inequality \cite{H91}, then $f$ admits some probability $\nu \in \mathcal{M}_{1}(f)$ with $\lambda(\nu)>0$. Take a probability $\mu\in \mathcal{M}_{1}(f)$ with $\lambda(\mu)=0$, define 
$$\nu_n:=\frac{1}{n}\nu+\Big(1-\frac{1}{n}\Big)\mu.$$
For any continuous observable $\varphi:\mathbb{S}^{1}\rightarrow\mathbb{R}$:
$$\int\varphi d\nu_n=\frac{1}{n}\int \varphi d\nu +\Big(1-\frac{1}{n}\Big)\int\varphi d\mu \xrightarrow[n\rightarrow\infty]{} \int\varphi d\mu,$$
hence $\lim\nu_n = \mu$. Furthermore, 
\begin{align*}
    \lambda(\nu_n)=\frac{1}{n}\lambda(\nu)+\Big(1-\frac{1}{n}\Big)\lambda(\mu)=\frac{1}{n}\lambda(\nu)>0.
\end{align*}
Note that in our context, $f$ has the specification property, hence $
\overline{\mathcal{M}_{e}(f)} = \mathcal{M}_{1}(f)$ (see \cite{S74}). Thus, $\mathcal{U}$ is dense.

\

Since  $\mathcal{U}$ is dense, by Morris's theorem \ref{morris} we have that $\mathcal{R}$
is dense in $C(\Sc^{1} , \R)$. 
\end{proof}

 We define now
 $$\mathcal{H}:=\mathcal{R}\cap(-\mathcal{R}).$$

Note that $\mathcal{H}$ is open and dense in $C(\Sc^{1} , \R)$ and if $\phi \in \mathcal{H}$ then:

\begin{enumerate}
    \item $t\phi$ is expanding for all $t \in \R;$
    \item $t\phi$ is not cohomologous to a constant for all $t \in \R\setminus\{0\}.$
\end{enumerate}


\

\begin{proof}[Proof of the Theorem \ref{maintheogoo}]
Note that if $\phi \in \mathcal{H}$ is H\"older continuous, 
then we can apply the Lemma \ref{GapAnalt} in $t\phi$, for all $t \in \R$.  Thus $\phi $ has no thermodynamic phase transition, and the topological pressure function $\R \ni t \mapsto P_{top}(f , t\phi)$ is strictly convex. 
\end{proof}

\

\subsection{Thermodynamic and spectral phase transitions}

In this section, we will prove the Theorems \ref{maintheA} and \ref{maintheorB}, as well as the Corollary \ref{maincoroA}.

\begin{proof}[Proof of the Theorem \ref{maintheA}]
(1) $\Rightarrow $ (2)] Suppose that $\phi$ has no thermodynamic phase transition. Fix $t \in \R$ and $\psi := t \phi$. Note that $\psi$ is expanding. In fact, if $\psi$ were not expanding, then $\psi$ would not be hyperbolic. Applying the Proposition \ref{propha}, we would have that $\psi$ has a thermodynamic phase transition, which is absurd. Thus $\psi$ is expanding, in fact $t\phi $ is expanding for all $t \in \R$. Remember that by Corollary \ref{corhyexp} we have $t\phi$ hyperbolic, for all $t \in \R$. Follow from the Proposition \ref{propexp} that $\mathcal{L}_{f, t\phi|E}$ has the spectral gap property, for all $t \in \R$.\\

(2) $\Rightarrow $ (1)] This follows directly from the Lemma \ref{Lemapress}.\\

(1) $\Leftrightarrow $ (3)] Suppose now that $\phi$ is not cohomologous to a constant. It follows by the proof of the equivalence between items (1) and (2) that $\phi$  has no thermodynamic phase transition if and only if $t\phi$ is expanding and hyperbolic for all $t \in \R$. Thus, applying the Proposition \ref{propexp} and Lemma \ref{Lemapress}, we have that the items (1) and (3) are equivalent.\\

Finally, applying item (1) from Lemma \ref{Lemapress}, the absence of a spectral phase transition for $\phi$ implies the uniqueness of the equilibrium state of $t\phi$, for all $t \in \R$.
\end{proof}

\

\begin{proof}[Proof of the Theorem \ref{maintheorB}] Suppose that $\phi $ is a H\"older continuous potential that has a phase transition. In particular, $\phi$ is not cohomologous to a constant. Remember initially that given $t \in \R$ we have that $t\phi$ is expanding if and only if $t\phi$ is hyperbolic, by Corollary \ref{corhyexp}.\\

(1)] Define $A := \{t \in \R : t\phi \text{ is expanding } \}$. 
Since $0 \in A$ and being expanding is an open property, then $A \subset \R$ is an open subset containing $0$. If $A = \R$ then the transfer operator $\mathcal{L}_{f, t\phi|C^{\alpha}}$ would have the spectral property for all $t \in \R$, by Proposition \ref{propexp}. In particular, $\R \ni t \mapsto P_{top}(f , t\phi)$ would be analytic by Lemma \ref{GapAnalt}. Since $\phi$ has a phase transition, by hypothesis, there is $\tilde{t} \neq 0$ such that $\tilde{t}\phi$ is not hyperbolic. Applying the Proposition \ref{propha}, if $\tilde{t} > 0$ then $t\phi$ is not hyperbolic for all $t \geq \tilde{t}$, otherwise $\tilde{t} < 0$ and $t\phi$ is not hyperbolic for all $t \leq \tilde{t}$.
Thereby, $A$ is an interval containing $0$.
In particular, there are $\hat{t}_{1} \in [-\infty , 0)$ and $\hat{t}_{2} \in (0 , +\infty]$ such that $\hat{t}_{1}\phi$ and $\hat{t}_{2}\phi$ are not hyperbolic, and $t_{1} \neq -\infty$ or $t_{2} \neq +\infty$. By Proposition \ref{propha} the topological pressure function $\R \ni t \mapsto P_{top}(f , t\phi)$ is linear in $[-\infty , t_{1}]$ and $[t_{2} , +\infty]$, furthermore $t\phi$ is expanding for all $t \in (t_{1}, t_{2})$. Applying the Proposition \ref{propexp} and item (2), (3) of the Lemma \ref{Lemapress}, the function $\R \ni t \mapsto P_{top}(f , t\phi)$ is analytic and strictly convex in $(t_{1},t_{2})$.\\

(2)] Since  $t\phi$ is expanding, for all $t \in (t_{1}, t_{2})$, follow of the Proposition \ref{propexp} that $\mathcal{L}_{f,\phi|E}$ has the spectral gap property. On the other hand, since  $\R \ni t \mapsto P_{top}(f , t\phi)$ is linear in $[-\infty , t_{1}]$ and $[t_{2} , +\infty]$, follow of the item (3) of Lemma \ref{Lemapress} that $\mathcal{L}_{f,t\phi|E}$ has no spectral gap property for $t \in [-\infty , t_{1}] \cup [t_{2} , +\infty]$.
\end{proof}





\

\begin{proof}[Proof of the Corollary \ref{maincoroA}]
Suppose that $f$ is Manneville-Pomeau-like and $\phi$ is a H\"older continuous potential. Suppose,  by contradiction, that $\phi$ has two phase transitions. By item (1) of the Theorem \ref{maintheorB}, there are $ t_{1} < 0$ and $t_{2} > 0$ such that $t \mapsto P_{top}(f , t\phi)$ is not analytic in $t_{1}$ and $t_{2}$. In fact, $t_{1}\phi$ and $t_{2}\phi$ are not hyperbolic. In particular, by item (1) of the Proposition \ref{propha}, there are $\eta_{1}$ maximizing probability of $-\phi$ and $\eta_{2}$ maximizing probability of $\phi$ with $\lambda(\eta_{1}) = \lambda(\eta_{2}) = 0$. 
Since $\log|Df(x)| \geq 0$ for all $x \in \Sc^{1}$ and $\log |Df(x)| = 0$ if and only if $x  = 0$ we have that $\delta_{0}$ is the unique $f$-invariant and ergodic probability with zero Lyapunov exponent.




Follow that $\eta_{1} = \eta_{2}$. Thus $\mathcal{M}_{1}(f) \ni \mu \mapsto \int \phi d\mu$ is a constant function. Applying \cite{T10}, we have that $\phi$ is cohomologous to a constant, which is absurd.
\end{proof}

\

\subsection{Large deviations principle}

In this section, we will prove the Corollary \ref{maincoroB}.

We will fix a H\"older continuous potential $\phi $. Define 
$$t_{1}:= \sup\{t < 0 : t\phi \text{ is not expanding }\} \;\text{ and }\; t_{2}:= \inf\{t > 0 : t\phi \text{ is not expanding }\}.$$ 
It follows from Theorem \ref{maintheA}, we have that $\phi$ has no phase transition if and only if $t_{1}= -\infty$ and $t_{2} = +\infty$. Moreover, $\mathcal{L}_{f,\phi|E}$ has the spectral gap property for all $t \in (t_{1}, t_{2}).$

As already discussed in section \ref{sect:LDR}, the specification property implies
$$
LD_{\phi, a, b} := \lim\frac{1}{n}\log \mu_{0}\Big(x \in \Sc^{1} : \frac{1}{n}\sum_{i=0}^{n-1}\phi (f^{i}(x)) \in [a ,b] \Big) =
$$
$$-h_{top}(f) + \sup\{h_{\nu}(f) : \int \phi d\nu \in [a , b]\}.
$$
On the other hand, it is known that one of the ways to calculate $LD_{\phi, a, b}$ is through the spectral gap property, given the probabilistic/functional approach derived from Gartner-Ellis’s theorem (see e.g. \cite{BCV16}).
We remember the highlight of this approach.

Define the free energy  $$\mathcal{E}_{\phi}(t):= \limsup_{n \mapsto +\infty} \frac{1}{n}\log \int e^{tS_{n}\phi }d\mu_{0},$$ where $t \in \R$ and $S_{n}\phi := \sum_{i=1}^{n-1}\phi \circ f^{i}$ is the usual Birkhoff sum.  Since  $\mathcal{L}_{f,t\phi|E}$ has the spectral gap property for all $t \in (t_{1}, t_{2})$, then the limit above does exist for $t \in (t_{1}, t_{2})$ and 
$$
\mathcal{E}_{\phi}(t) = \log \rho(\mathcal{L}_{f,t\phi|E}) - \log \rho(\mathcal{L}_{0|E}) = P_{top}(f, t\phi) - h_{top}(f).
$$

Suppose that $\phi$ is cohomologous to a constant. In this case, $\phi$ has no phase transition and $S_{\phi} = \{\int\phi d\mu_{0}\}.$

Suppose then that $\phi$ is not cohomologous to a constant. Follow from the Proposition \ref{propexp} and Lemma \ref{Lemapress} that $\mathcal{E}_{\phi}$
is analytic and stricly convex in $(t_{1} , t_{2})$. Hence, it is well defined the "local"
Legendre transform $I$ given by
\begin{equation*}
I(s)
	:= \sup_{t_{1} < t < t_{2}} \; \big\{ st-\cE_{\phi}(t) \big\}.
\end{equation*}

The domain of $I$ will be $\{\mathcal{E}_{\phi}'(t) : t \in (t_{1}, t_{2})\} = \{\int \phi d\mu_{t\phi} : t \in (t_{1}, t_{2})\}$, and $I$ will be analytic, strictly convex and non-negative. 
In fact, it is not hard to check the variational property:
$
I (\cE_{\phi}'(t) )
	= t \cE_{\phi}'(t) - \cE_{\phi}(t).
$
Moreover, $I(s)=0$ if and only if $s=\int\phi d\mu_{0}$. 

Define $\lambda_{\min} := \inf_{t \in (t_{1}, t_{2})}\int \phi d\mu_{t\phi}$ and $\lambda_{\max} := \sup_{t \in (t_{1}, t_{2})}\int \phi d\mu_{t\phi}$, where $\mu_{t\phi}$ is the equilibrium state of $f$ with respect to $t\phi$ obtained in Lemma \ref{GapAnalt}. If $t_{1} \in \R$ then take $\eta_{1} \in \mathcal{M}_{1}(f)$ a maximizing measure of $-\phi$ with zero entropy, and define $\beta_{\min} := \int \phi d\eta_{1}$. Otherwise, define $\beta_{\max}:= \lambda_{\min}$. Analogously, we can define $\beta_{\max}$. Applying \cite{T10}, we have that $S_{\phi} = \{\int \phi d\nu : \nu \in \mathcal{M}_{1}(f)\}$. Thus, we conclude that 
    $S_{\phi} = [\beta_{\min} \,,\, \beta_{\max}] \supset [\lambda_{\min}\,,\, \lambda_{\max}] $ and the domain of $I$ will be $(\lambda_{\min}\,,\, \lambda_{\max})$.

As a consequence
of the differentiability of the free energy function,
applying the Gartner-Ellis's theorem (see e.g.~\cite{DZ98,RY08}), we have that
$$
LD_{\phi, a,b}
	=-\inf_{s\in[a,b]} I(s),
$$
for all interval $[a,b]\subset (\lambda_{\min} , \lambda_{\max})$.
In particular, in our context, 
$$I(s) = h_{top}(f) - \sup\{h_{\nu}(f) : \int\phi d\nu = s\}.$$

\begin{lemma}
 \begin{itemize}
     \item[(i)] If $t_{1} \in \R$ then for all $[a , b] \subset [\beta_{\min} \,,\, \lambda_{\min} ]$ we have $LD_{\phi,a,b} = -h_{top}(f) + t_{1}(\beta_{\min} - a);$\\
     \item[(ii)] If $t_{2} \in \R$ then for all $[a , b] \subset [\lambda_{\max}\,,\, \beta_{\max} ]$ we have $LD_{\phi,a,b} = -h_{top}(f) + t_{2}(\beta_{\max} - b);$
 \end{itemize}
 \end{lemma}
 \begin{proof}
 
 [(i)] We can assume without loss of generality that $\beta_{\min} \neq \lambda_{\min}$. Let $d \in [\beta_{\min} \,,\, \lambda_{\min} ]$ be and $\nu$ be an $f$-invariant probability such that $\int\phi d\nu = d$. 
 Then $h_{\nu}(f) + t_{1}\int\phi d\nu \leq P_{top}(f, t_{1}\phi) = t_{1}\beta_{\min} \Rightarrow h_{\nu}(f) \leq t_{1}(\beta_{\min} - d)$. 
 On the other hand, let $\eta_{1} $ be a maximizing measure of $-\phi$ with zero entropy, and $\mu_{t_{1}}$ be an accumulation point of $\mu_{t\phi}$ when $t$ converges to $t_{1}$.
 Define 
 $$\nu_{d} := \frac{d - \beta_{\min}}{\lambda_{\min} -\beta_{\min}}\mu_{t_{1}} + \Big(1-\frac{d - \beta_{\min}}{\lambda_{\min} -\beta_{\min}}\Big)\eta_{1}.$$ Thus $\int\phi d\nu_{d} = d $ and $h_{\nu_{d}}(f) = t_{1}(\beta_{\min} - d).$
 Therefore, we conclude that $LD_{\phi,a,b} = -h_{top}(f) + t_{1}(\beta_{\min} - b).$

 \
 
 [(ii)] Analogously, unless you change $\phi$ by $-\phi$, we will prove the item (ii).
 \end{proof}
 \

\begin{proof}[Proof of the Corollary \ref{maincoroB}]
Note initially that, taking into account the previous discussions, the function $S_{\phi} \ni a \mapsto LD_{\phi,a,a}$ is convex. Moreover, this function is affine in $[\beta_{\min}\,,\, \lambda_{\min}] \cup [\lambda_{\max} \,,\, \beta_{\max}]$, and analytic, strictly convex, non-negative in $(\lambda_{\min} \,,\, \lambda_{\max})$. Finally, $LD_{\phi,a,b} = \inf_{s\in [a , b]} LD_{\phi,s,s}$.
Thus, defining 
$$
\Delta_{1} := \Big\{(a, b) \in \Delta : b < \lambda_{\min} \text{ or } a > \lambda_{\max}\Big\}$$
$$
\Delta_{2} := \Big\{(a, b) \in \Delta : \int\phi d\mu_{0} \in [a,b]\Big\}$$
$$
\Delta_{3} := \Big\{(a, b) \in \Delta : \lambda_{\min} \leq a \leq b < \int\phi d\mu_{0} \text{ or } \int\phi d\mu_{0} < a \leq b \leq \lambda_{\max}\Big\},$$
we conclude the proof of the corollary.
\end{proof}

\

\subsection{Multifractal analysis}

In this section, we will prove the Corollary \ref{maincoroC}.

We will fix a H\"older continuous potential $\phi $ and the notations from the previous section. We are interested in understanding the sets: 
$$
X_{\phi , a , b}:= \Big\{x \in \Sc^{1}: a \leq \liminf \frac{1}{n}\sum_{i=0}^{n-1}\phi (f^{i}(x)) \leq \limsup \frac{1}{n}\sum_{i=0}^{n-1}\phi (f^{i}(x)) \leq b\Big\}.
$$

The proof could be inspired by the ideas of \cite{BCF22}, 
which indicate that a good understanding of phase transitions implies a good understanding of multifractal analysis. However, since we are in the one-dimensional context, we will use Hofbauer's work \cite{Ho10}.
Hofbauer introduces the pressure function $\tau(t):=P_{top}(f, t\phi)$ and its Legendre transform $\hat{\tau}$ defined by \begin{equation}\label{LegTransf}
     \hat{\tau}(s)=\inf _{t \in \mathbb{R}}(\tau(t)-s t)
 \end{equation} on the set $H := \{s \in \R : st \leq \tau(t) \text{ for all } t \in \R\}$. 
In his paper, Hofbauer uses the Legendre transform $\hat{\tau}$ to characterize the entropy spectrum of Birkhoff averages. In fact, it has been proven that 
$$h_{X_{\phi , a , b}}(f) =\max _{s \in H \cap[a, b]} \hat{\tau}(s).$$

\

\begin{proof}[Proof of the Corollary \ref{maincoroC}]
It is enough to relate Hofbauer's work to our previous context:
\begin{itemize}
    \item $H = [\beta_{\min}\,,\, \beta_{\max}] = S_{\phi}$.
    \item If $s \in [\beta_{\min}\,,\, \lambda_{\min}]$ then $\hat{\tau}(s) = s(\beta_{\min}  - a)$.
    \item If $s \in [\lambda_{\max}\,,\, \beta_{\max}]$ then $\hat{\tau}(s) = s(\beta_{\max}  - a)$.
\end{itemize}
Moreover, If $s \in (\lambda_{\min}\,,\, \lambda_{\max})$ then $$\hat{\tau}(s) = \inf _{t \in \mathbb{R}}\{P_{top}(f,t\phi)-s t\} = \inf _{t \in (\lambda_{\min}\,,\,\lambda_{\max})} \{P_{top}(f,t\phi)-s t\} = 
$$
$$h_{top}(f) - \sup _{t \in (\lambda_{\min}\,,\,\lambda_{\max})}\{st - \mathcal{E}_{\phi}(t)\}= h_{top}(f)  -  I(s).$$

Thus, given $(a,b) \in \Delta$ we have:
$$h_{X_{\phi , a , b}}(f) =\max _{s \in H \cap[a, b]} \hat{\tau}(s) = h_{top}(f) - \inf_{s \in [a,b]}I(s) =  h_{top}(f) + LD_{\phi,a,b}.$$

Applying the Corollary \ref{maincoroB}, we conclude the proof of the Corollary \ref{maincoroC}.
\end{proof}

\section{Further comments and questions}\label{stratpro}

\subsection{Differentiability of the topological pressure function}

Suppose that $f : \Sc^{1} \rightarrow \Sc^{1}$ is Manneville-Pomeau-like. By Pianigiani's work \cite{Pi80}, if $f$ is a $C^{2}$-local diffemomorphism then $f$ does not admits a finite a.c.i.p.. In particular, taking the geometric potential $\phi = -\log |Df|$, we have that $\R \ni t \mapsto P_{top}(f, t\phi)$ is not analytic, but it is differentiable (see \cite{BC21}). Therefore, the absence of the analyticity of the topological pressure function does not imply the lack of differentiability, even if the potential is regular.  On the other hand, if $f$ is a $C^{1+ \alpha}$-local diffeomorphism on a neighbourhood of the indifferent fixed point and $f$ is not $C^{2}$ then $f$ admits a unique finite a.c.i.p (see \cite{Pi80}). It follows from the \cite{BC21} that $\R \ni t \mapsto P_{top}(f, t\phi)$  is not differentiable in $t = 1$. 

Thus, we proposed the following question:

\begin{question}
    If $f : \Sc^{1} \rightarrow \Sc^{1}$ is a  $C^{1+\alpha}$-local diffeomorphism then $f$ is expanding or does there exists a H\"older continuous potential $\phi : \Sc^{1} \rightarrow \R$ such that the topological pressure function $\R \ni t \mapsto P_{top}(f,t\phi)$ is not differentiable ?
\end{question}

\subsection{Uniqueness of the equilibrium states}\label{subsec:G}

The Theorem \ref{maintheA} shows that if the potential $\phi$ has no phase transition, then $t\phi$ has a unique equilibrium state for all $t \in \R$. Note that the reciprocal is not true. In fact, 
for each $\alpha \in [0 , 1]$ define the constants $$
b=b(\alpha):=\left((\frac{1}{2})^{3+\alpha}-\dfrac{4+\alpha}{4+2\alpha}\Big(\frac{1}{2}\Big)^{2+\alpha}\right)^{-1}, \,\, a=a(\alpha):=\dfrac{-b(4+\alpha)}{4+2\alpha}$$ and the polynomial $g_{\alpha} : [0 , \frac{1}{2}] \rightarrow [0 , 1]$ given by $g_{\alpha}(y)=y+ay^{3+\alpha}+by^{4+\alpha}$. We can then define a family of Manneville-Pomeau-like: 
\begin{equation}\label{difeo}
 f_{\alpha}(y)=\begin{cases}
 g_{\alpha}(y), \text{ if } 0\leq y \leq 1/2 \\
 1-g_{\alpha}(1-y), \text{ if } 1/2 < y \leq 1.
 \end{cases}
\end{equation}


Each family member will be a $C^{2}$-local diffeomorphism. It follows from the \cite{Pi80} that the geometric potential $\phi := -\log |Df_{\alpha}|$ has a unique equilibrium state. On the other hand, it follows from the \cite{BC21} that $f_{\alpha}$ has a phase transition with respect to $\phi$.


Suppose now that $f : \Sc^{1} \rightarrow \Sc^{1}$ is a transitive $C^{1+\alpha}$-local diffeomorphism. By \cite{BC21}, we know that $f$ is expanding if and only if $f$ has a phase transition with respect to the geometric potential $\phi = -\log |Df|$. With a view to the previous discussion and \cite{W92}, we propose the following question:

\begin{question}
    If $f : \Sc^{1} \rightarrow \Sc^{1}$ is a $C^{1+\alpha}$-local diffeomorphism then $f$ is expanding or does there exists a H\"older continuous potential $\phi : \Sc^{1} \rightarrow \R$ such that $\phi$ has at least two equilibrium states ?
\end{question}

\subsection{Number of phase transitions}

In our setting, Theorem \ref{maintheorB} guarantees that the maximum number of phase transitions is two. However, every example known of phase transitions for regular potentials occurs with respect to the geometric potential  $\phi := -\log |Df|$. In fact, it has a unique phase transition (see e.g. \cite{Lo93,PS92,BC21}). Thus, we propose the following question:

\begin{question}
Is there  H\"older continuous potential $\phi : \Sc^{1} \rightarrow \R$
    such that the topological pressure function $\R \ni t \mapsto P_{top}(f,t\phi)$ has two phase transitions ?
\end{question}


\subsection{Higher dimensional case}

In the proof of our main theorems, we used the fact that our dynamics are topologically conjugate to an expanding dynamics. Thus, a natural question is whether we can obtain similar results for local diffeomorphisms that are topologically conjugate to an expanding dynamics. In particular, we are going to discuss this question for an explicit family of dynamics.

Fix $x_{1} < x_{2} < \ldots < x_{k} = x_{1}$ points in $\Sc^1$. Define 
$V(\{x_{j}\}_{j = 1}^{k})$ as the set of  maps $f : \Sc^1 \rightarrow \Sc^1 \text{ transitive } C^{r}\text{-local diffeomorphism such that }
f_{|(x_{j} , x_{j+1})
} \text{ is injective, } 
$
$f([x_{j} , x_{j+1}]) = \Sc^1 , f(x_{j}) = x_{1}, Df(x_{1}) = 1, Df(x_{j}) \geq 1 \text{ and } |Df(x)| > 1
\text{ for all } x \neq x_{j} \text{ and } j = 1 , \ldots, k .$

Let $$F : \mathbb{T}^{d} \times \Sc^1 \xrightarrow[(x , y) \mapsto  (g(x) , f_{x}(y)) ]{} \mathbb{T}^{d} \times \Sc^1$$ be  a $C^{r}$-skew-product, $r>1$ integer, such that $g$ is an expanding linear endomorphism and each $f_{x} \in V(\{x_{j}\}_{j = 1}^{k})$. In \cite{BCF22}, it is shown that $F$ is topologically conjugate to an expanding dynamics. 

\begin{example}
For each $\alpha \in [0 , 1]$, let $f_{\alpha}$ the Manneville-Pomeau-like map defined in Section \ref{subsec:G}.
Observe that $f_{\alpha} \in V(\{0 , \frac{1}{2}, 1\})$ for all $\alpha \in [0 , 1]$. Therefore, we can define the intermittent skew product 
$$F : \mathbb{T}^{d} \times \Sc^1 \xrightarrow[(x , y) \mapsto  (g(x) , f_{x}(y)) ]{} \mathbb{T}^{d} \times \Sc^1,$$ where $g$ is an expanding linear endomorphism.
\end{example}

By \cite{BCF22}, $F$ has a thermodynamic and spectral phase transition with respect to geometric-like potential $\phi^{c}(x , y) := \log |\frac{\partial F}{\partial y}(x,y)|$. More formally:

\begin{theorem}[\cite{BCF22}]
There exists $t_{0} \in (0 , 1]$ such that:\\

(i) the topological pressure function $\R \ni t \mapsto P_{top}(F , t\phi^{c}) $ is analytic, strictly decreasing and strictly convex in $(-\infty , t_{0})$ and constant equal to $h_{top}(g)$ in $[t_{0} , +\infty)$;\\

(ii) the transfer operator $\mathcal{L}_{F, t\phi^{c}}$ has the spectral gap property for all $t < t_{0}$ and has no the spectral gap property for all $t \geq t_{0}$, acting on $C^{r-1}(\mathbb{T}^{d} \times \Sc^1 , \C)$.

\end{theorem}

However, we are still determining what happens to this class of dynamics with respect to other regular potentials. Therefore, we propose the following questions:

\begin{question}
(i) Given a $C^{r-1}$-potential $\phi$, can we describe the topological pressure function $t\mapsto P_{top}(F,t\phi)$?

(ii) Is there a dense subset of $C^{r-1}$-potentials such that $F$ has no thermodynamic phase transition with respect to its potentials? 

(iii) Is there a dense subset of $C^{r-1}$-potentials such that $F$ has a unique equilibrium state with respect to its potentials?

(iv) Is there a dense subset of $C^{r-1}$-potentials such that $\mathcal{L}_{F, \phi}$ has the spectral gap property, acting on $C^{r-1}(\mathbb{T}^{d} \times \Sc^1 , \C)$? 
\end{question}


\vspace{.3cm}
\subsection*{Acknowledgements}
This work is part of the second author's PhD thesis at the Federal University of Bahia. TB was partially supported by CNPq (Grants PQ-2021) and CNPq/MCTI/FNDCT project 406750/2021-1, Brazil.  AF was supported by CAPES-Brazil. The authors are deeply
grateful to Paulo Varandas for their useful comments. We thank the anonymous referee for the careful reading of the manuscript.


\bibliographystyle{alpha}

\end{document}